\documentclass[options]{amsart}
\usepackage{latexsym,ifthen,amssymb}
\usepackage[toc,page,title,titletoc,header]{appendix}
\usepackage{enumerate}
\usepackage{graphicx}
\usepackage{bm}
\usepackage{subfigure}
\usepackage{float}

\usepackage{tikz-cd}
\usepackage{caption}

\usepackage{multirow}
\usepackage{mathrsfs}

\usepackage{hyperref}

\usepackage{cleveref}
\newtheorem{theorem}{Theorem}[section]
\newtheorem{lemma}[theorem]{Lemma}

\newtheorem{claim}[theorem]{Claim}

\newtheorem{proposition}[theorem]{Proposition}

\theoremstyle{definition}

\theoremstyle{remark}

\numberwithin{equation}{section}

\theoremstyle{noparens}
\newtheorem*{question*}{Question}
\newtheorem*{theorem*}{Theorem}

\def\uhr{\upharpoonright}
\def\h{\hat}
 \def\xor{\Delta }
  \def\NR{NR}
 
 \title{There is no maximal $K$-degree}

 \author{Lu Liu}
\address{School of Mathematics and Statistics,
Central South University,
City Changsha, Hunan Province,
China. 410083}
\thanks{Author is supported by  NSFC grant 12471002.}
\email{g.jiayi.liu@gmail.com}
\subjclass[2010]{Primary 03D32; Secondary 03D30, 68Q30}
\keywords{computability theory, algorithmic randomness theory,  Kolmogorov complexity,
$K$-degree}

\begin{document}
\maketitle
\begin{abstract}
The Kolmogorov complexity of a string characterize how complex it is 
to describe the string. If every prefix of a real $x$ is 
more complex to describe than every prefix
(of the same length) of real $y$, then it is seen as $x$ is more 
complex to describe than $y$.
It is wondered if there is a real $x$ so that   no other reals
 are strictly more complex (to describe) than $x$.
The behavior of 
Kolmogorov complexity functions generated by reals 
(namely $n\mapsto$ the minimal description length of the real)
 is quite chaos and seems non free.
 Therefore, it is widely believed that 
 there are many reals that are maximally complex to describe.
 For instance, it is conjectured that all random enough reals
 have maximal $K$-degree. 
 In this paper, it is shown that there is no real with maximal $K$-degree.
Actually, for almost all real $x$,
we can uniformly computably  find another real $y$ whose 
$K$-degree is strictly
above that of $x$. 
The same construction  also derives the following:
(a) There is no maximal   $C $-degree.
(b) Every $n$-random real $x$ admits an $n$-random real $y$
such that  $y>_Kx$ and $y>_C x$; moreover, the difference of $K$-complexity 
and $C$-complexity
of $y,x$ (namely $K(y\uhr n)-K(x\uhr n)$ and $C(y\uhr n)-C(x\uhr n)$) do  not tend to infinity.

\end{abstract}

\section{Introduction}
Kolmogorov complexity assigns to a finite object the length of its shortest
effective description.  The underlying invariance principle goes back to
Kolmogorov, Solomonoff, and Levin; the prefix-free version, 
was developed by Levin and Chaitin
\cite{Solomonoff64,Kolmogorov65,ZL70,Chaitin75}.  Fix a universal prefix-free
machine $U$, and let
\[
 K(\sigma)=\min\{|p|:U(p)=\sigma\}.
\]
   Replacing $U$ by another optimal prefix-free
machine changes $K$ by at most an additive constant.  Prefix-free complexity
also provides a fundamental bridge between finite descriptions and randomness
of infinite sequences: by the Levin--Schnorr theorem, a real $x\in 2^\omega$
is Martin-L\"of random if and only if  \cite{ML66,Schnorr71,Chaitin75}
$$
 K(x\uhr n)\geq n-O(1).
$$
  We write \emph{1-random}
    for Martin-L\"of random.  Standard references for  background in computable randomness theory include
\cite{Nies09,DH10}.

The sequence $n\mapsto K(z\uhr n)$ contains substantially finer information
than the single assertion that $z$ is random.  For every $\sigma$ one has
$K(\sigma)\leq |\sigma|+K(|\sigma|)+O(1)$, but even for random $z$ the distance
of $K(z\uhr n)$ from the natural upper and lower bounds oscillates.  Precise
forms of these upward and downward oscillations were obtained by Miller and Yu
\cite{MY11}.  At the stronger randomness level, Miller proved that $z$ is
2-random exactly when its prefix-free initial-segment complexity is infinitely
often maximal, in the sense that
\[
 K(z\uhr n)\geq n+K(n)-O(1)
\]
for infinitely many $n$ \cite{Miller09}.  These results illustrate both the
strength and the irregularity of comparisons based on all initial segments.

For $z,\h z\in 2^\omega$, define
\[
 z\leq_K \h z \quad\Longleftrightarrow\quad
  \ K(z\uhr n)\leq K(\h z\uhr n)+O(1).
\]
Thus $\h z$ is at least as complex as $z$, up to an additive constant, at
every initial-segment length.  Write $z=_K \h z$ when both $z\leq_K \h z$ and
$\h z\leq_K z$, and write $z<_K \h z$ when $z\leq_K \h z$ but $\h z\not\leq_K z$.
The relation $\leq_K$ is obviously a partial order and $=_K$ a equivalence relation.
Thus the relation $=_K$ give rise a degree notion, namely 
let $[z]_K$ denote the $=_K$-equivalence class $\{\h z\in 2^\omega: \h z=_K z\}$.
This is referred as
\emph{$K$-degree} of $z$.  This way of comparing reals is implicit in Solovay's work
\cite{Solovay75} and was introduced and systematically studied by Downey,
Hirschfeldt, and LaForte \cite{DHL04}.  The least $K$-degree consists of the
$K$-trivial reals, namely those $z$ satisfying
$$K(z\uhr n)\leq K(n)+O(1).$$  Beyond the least degree, the partial order $\leq_K$ has a rich  
structure.  
Yu, Ding and Downey \cite{yu2004kolmogorov} showed that there is no greatest $K$-degree. 
Meanwhile the 1-random $K$-degree
cannot be minimal  \cite{MY11}. 
Barmpalias and Vlek \cite[Theorem~5.6]{BV11}
 constructed reals $z$ that are sufficiently beyond $K$-trivial
but (whose $K$-degree) are far below 1-random such that 
there are continuum many reals with $K$-degree strictly below that of $z$.
Despite being a rich structure, the $K$-complexity function
of a real $z$ (namely $n\mapsto K(z\uhr n)$)
cannot be chosen freely. For example,
 if $z_0\oplus z_1$ is 1-random, then
$z_0$ and $z_1$ have no common $K$-upper bound \cite{MY08}; 
there are comparable 1-random $K$-degrees, and the cones inside the random
degrees can have   different cardinalities \cite{MY11}. 
  Small upper cones above highly
random reals---the cone above every 2-random real is countable
\cite{Miller09}.
This result seems to indicate that it is hard to construct reals with 
$K$-degree strictly above a 2-random real.
 Therefore \cite{Miller09} conjectured  that all 2-random reals have maximal $K$-degree.
Here
 a $K$-degree $[z]_K$ is \emph{maximal} if there is no real $\h z\in 2^\omega$ such that
$z<_K \h z$.  Miller and Nies \cite[Question~9.7]{MN06} asked if there exists a maximal
$K$-degree.
The main result of this paper is to show that maximal $K$-degree does not exist.
i.e., Every real $z$ admits a $K$-degree strictly above that of $z$.
\begin{theorem}
\label{maxKth0}
There is no maximal $K$-degree.
\end{theorem}
Moreover,
on the measure-one class
of 1-random reals, 
we can uniformly computably find a real with strictly higher $K$-degree.
\begin{theorem}\label{maxKth1}
There are $k^*\in\omega$ and total \footnote{Meaning, the Turing functional is
total on every oracle.} Turing functionals
$\Psi_1,\ldots,\Psi_{k^*}$ such that, for every 1-random real
$z\in 2^\omega$,  
$z<_K\Psi_{\h k}(z)$ for some $\h k\leq k^*$.
\end{theorem}
Theorem \ref{maxKth1} is quite counter intuitive
since usually, $\h z$ being computable from $z$
indicates that $\h z$ is more simple than $z$.
The same argument also gives
\begin{theorem}\label{maxKth:maxCdegree}
There is no maximal $C$-degree.
\end{theorem}
\def\MLR{MLR}
\def\tteq{\equiv_{tt}}
The construction of $\Psi_k$ makes it
an invertible function (in a very strong sense),
Combine with this fact it also gives the following results. 
Miller and Nies \cite[Question~9.6]{MN06}, asked whether two distinct $K$-degrees represented
by $2$-random reals can be comparable;
and \cite[Question~9.10]{MN06} asked whether 
two distinct $C$-degrees represented by 1-random reals can be comparable.
 Here a real   is $2$-random iff it is 1-random relative to 
$\emptyset'$. It will be shown that there are many such comparable pairs.
Miller and Yu \cite{MY11} asked if $x<_K y$
implies that the difference of their $K$-complexity
must tends to infinite.
It will be shown that this implication is not true.
Moreover, all these requirements can be satisfied together.
\begin{theorem}\label{maxKthm:relative}
For every oracle $d$ and every 1-random relative to $d$ real $x $, there is a 
 1-random relative to $d$ real $y\tteq x$ such that $x<_K y$ and $x<_C y$;
 moreover, for infinitely many $n\in\omega$,
$$
\big|K(x\uhr n)-K(y\uhr n)\big|, \big|C(x\uhr n)-C(y\uhr n)\big| = O(1).
$$ 
In particular, for every $n\geq1$, every $n$-random real has an
$n$-random strict $K$-upper bound in the same truth-table degree.
 
\end{theorem}

The rest of this paper will prove Theorem \ref{maxKth0}, \ref{maxKth1}, \ref{maxKthm:relative},
\ref{maxKth:maxCdegree}.

\ \\

\noindent\textbf{Intuition.}
Suppose   strings $x,y\in 2^m$ appears adjacently in $z$.
i.e.,
$$z= ...x\cdot y...$$
Moreover, suppose $x$ is simple (say $x=0^m$) and $y$ is complex.
Then swapping $x,y$ in $z$ to get:
$$
\h z= ...y\cdot x...
$$
 will move forward the complexity peak.
More specifically, the $K$-complexity function of $z$ should
reach a low peak at the end of $x$ in $z$
(position $n_0$).
On the other hand, the $K$-complexity function of $\h z$
should reach a high peak at $n_0$ thus exceeding that of 
$K(z\uhr n_0)$.

\noindent\textbf{Organization.}
In section \ref{maxKsecframwork} we firstly describe the framework 
of the proof, a part of the construction of 
$(\Psi_{\h k}:\h k<k^*)$ and gives a key observation
(in more details than the above intuition).
In section \ref{maxKsecconstruct}
we concretely define the Turing functionals $\Psi_{\h k}$
and in section  \ref{maxKsecverify}
we verify that if none of the $K$-degree of 
$(\Psi_{\h k}(z):\h k<k^*)$  is strictly beyond that of $z$,
then the $K$-complexity function of $z$ is far below  $n\mapsto n$.
We end up by introducing some notations.

\noindent\textbf{Notation.}
We   write $\sigma\rho$
or $\sigma\cdot\rho$ (or occasionally $(\sigma,\rho)$) for the concatenation of $\sigma$ and $\rho$.
We write $\sigma\uhr n$ for the $n$-length prefix of $\sigma$
(if $n>|\sigma|$, then $\sigma\uhr n:= \sigma$).
 We write $O(1)$ for  absolute constant;
 write $O_v(1)$ for a constant depending on object $v$.
 
 \section{Proof of main results}
 
 We start with an intuitive description of the proof
 and highlight where the difficulties are.
\subsection{Framework and observation.}
\label{maxKsecframwork}
 We will define
 computable functional $$f_a: 2^\omega\rightarrow 2^\omega$$
  indexed on $a\in 2^k$   ($k=4$ will be enough) 
 such that:
 \begin{align}\label{maxKeq0}
 &\text{$z\leq_K f_a(z)$
  for all $z\in 2^\omega$ and $a\in 2^k$;}\\ \nonumber
  &\text{If $f_a(z)\leq_K z$ for all $a\in 2^k$, then $\lim\limits_{n\rightarrow\infty }K(z\uhr n) -n\rightarrow -\infty$.}
 \end{align}
Since the $K$-complexity of  every 1-random real satisfy 
$K(x\uhr n)-n \geq O(1)$, therefore if $f_a(z)\leq_K z$ for all $a\in 2^k$, then 
the $K$-degree of $z$ is strictly below that of any 1-random real.
 Thus, (\ref{maxKeq0}) concludes the proof of Theorem \ref{maxKth0}.
 
 \ \\
 
 \noindent\textbf{What is $f_a$ and how it works.}
  For a real $z\in 2^\omega$,
 we regard $z$ as a sequence of blocks
 $z_1z_2\cdots$ where $z_i\in 2^{2m_i}$
 with $(m_i\in\omega:i\geq 1)$ a computable sequence 
 of non decreasing integers (we take $m_i:\approx [\log_2 i]$).
 Each block $z_i$ is seen as a pair $(x_i,y_i)$
 where $x_i,y_i\in 2^{m_i}$.
  We write $z_{\leq i}$ to denote the 
 string $ z_1 \cdots  z_i$ and similarly 
 for $z_{<i}$, $z_{[i,j)} = 
  z_i\cdots z_{j-1} $ or $z_{(i,j)}$ etc.

 A key ingredient in the construction of $f_a$ is the following
 function $g_m:2^{2m}\rightarrow 2^{2m}$,
 for $x,y\in 2^m$, let 
 \begin{align}\label{maxdefg_m}
 g_m(x,y):= \left\{
 \begin{aligned}
 &(y,0^m)&\text{ iff }x=0^m\\  
 &\text{to be specified} &\text{   otherwise.}
 \end{aligned}
 \right.
 \end{align}
 We define  $g$ on $z$ by applying $g_{m_i}$ on the $i^{th}$ block of $z$. i.e.,
 $$g(z):=  g_{m_1}(x_1,y_1)\cdot g_{m_2}(x_2,y_2) \cdots.$$
 For $a=0^k$, $g$ is the functional $f_a$.

Let's see why the complexity deficiency occurs under 
hypothesis  $g(z)\leq_K z$.
Suppose  $z_{\leq i} = z_{<i}\cdot (0^{m_i},y)$.
We will make $g_m$ a permutation,
so the $|z_{<i}|$-length prefix of $g(z) $
(denoted as $\h z_{<i}$) is 
computable from $z_{<i}$ and $z_{<i} $ is also 
computable given $\h z_{<i}$.
This entails  $K(z_{<i}) \leq K(\h z_{<i})+O(1)$.
On the other hand,
using the hypothesis we have
\begin{align}\label{maxKkey}
K(z_{\leq i}) = K(z_{<i}\cdot (0^{m_i}, y)) &\leq K(\h z_{<i}\cdot y)+O(1)\\ \nonumber
&\leq K( z_{<i}\cdot 0^{m_i})+O_z(1)
\leq K( z_{<i})+O_z(1).
\end{align} 
 The first inequality is because
 $z_{<i}\cdot (0^{m_i},y)$ is computable 
 given $\h z_{<i}\cdot y$ 
 (since $(m_i:i\in\omega)$ is computable and $g_m$ is a permutation).
  The second inequality is due to $g(z)\leq_K z$. 
 
\ \\
 
 \noindent\textbf{What else do we need.}
 \begin{itemize}
 \item 
 We firstly need to complete the definition of $g_m$.
 In order to satisfy $z\leq_K f_a(z)$,
 we make $g_m$ a permutation on $2^{2m}$. Moreover,
  any $\rho\in 2^N$ admits at most $O(1)$ many $g$-pre-image.
  Here a $g$-pre-image of $\rho$ is a string $\sigma\in 2^N$
  such that $g([\sigma])\cap [\rho]\ne\emptyset$.
 This obviously suffices to entail $z\leq_K g(z)$.
 
 \item The deficiency in (\ref{maxKkey}) is the deficiency
 occurring around the block with a $0^m$-first component
  (call such block rooted).
 We need to push forward this deficiency towards
 a general position. However, for instance, what if $z$ does not 
 have any rooted. 
This is partly done by choosing $m_i$ to be approximately $[\log_2 i]$.
Such sequence $(m_i:i\in\omega)$ clearly makes it impossible to avoid
$0^m$-first component while being 1-random
since $$\prod_i (1-2^{-m_i}) = 0.$$
 However, this is not enough.
 To overcome the deficiency loss
  for a general position $z\uhr n$,
  we define more function $g$, namely $(f_a:a\in 2^k)$
  so that it makes $z$ to pay more complexity loss 
  when it avoid all rooted blocks (for $2^k$ many function $f_a$,
  the first component has to avoid $2^k$ many strings to avoid being 
  a rooted block).

 \end{itemize}

 \subsection{Construction of $f_a$.}
 \label{maxKsecconstruct}
 As mentioned in section,
 we firstly complete our definition (\ref{maxdefg_m}) of $g_m$
 to make it a permutation on $2^{2m}$
 and make $g$ having the $O(1)$-many-$g$-pre-image property.
 For $x,y\in 2^m$, regard $x,y$ as binary expansion of natural numbers,
 we write $x+y, x+n $ (where $n$ is a natural number) for
  $$x+y(\mod 2^m), x+n(\mod 2^m) \text{ respectively }$$
  where the result is also a binary string of the same length.
Similarly for $x-y,x-n$.
 Let
 \begin{align}
 g_m(x,y):= 
 \left\{
 \begin{aligned}
 &(y,0^m)&\text{ iff }x=0^m\\ \nonumber
 &(x+1,y+1)&\text{ iff }y<x\\ \nonumber
 &(x,y)&\text{ iff }y\geq x>0.
 \end{aligned}
 \right.
 \end{align}
  Let $m_i:= [\log_2 i]\vee k+1$.
 For $z =  z_1 z_2 \cdots $ where $z_i =  x_i y_i $ where $x_i,y_i\in 2^{m_i}$
 Let $$g(z):=  g_{m_1}(x_1,y_1)\cdot  g_{m_2}(x_2,y_2) \cdots.$$
 Obviously $g$ is computable.
 For $\rho,\sigma\in 2^n$, 
 we say $\rho$ is a \emph{$g$-pre-image} of $\sigma$
 iff $g([\rho])\cap[\sigma]\ne\emptyset$.
 
 \begin{claim}\label{maxKclaim0}
 We have:
 \begin{enumerate}
 \item $g_m$ is a permutation on $2^{2m}$.
 \item For every $\sigma\in 2^n$, $\sigma$ admits 
 at most five $g$-pre-image of $\sigma$.
 \end{enumerate}
 \end{claim}
 \begin{proof}
 For item (1): 
It suffices to show every element of $2^{2m}$ admits a pre-image 
 under $g_m$.

 \noindent For $(x,y) $ with $y> x=0$, 
 the pre-image is $(1^m,y-1)$. 

 \noindent For $(x,y)$ with $y=x=0$,
 the pre-image of $(x,y)$ is $(0^m,0^m)$.
 
  \noindent For $(x,y)$ with $x>y=0$, the pre-image of 
 $(x,y)$ is $(y,x)$. 
 
  \noindent For $(x,y)$ with $y\geq x>0$, the pre-image of $(x,y)$
 is $(x,y)$. 
 
 \noindent For $(x,y)$ with $x>y>0$, the pre-image of $(x,y)$
 is $(x-1,y-1)$.

 \noindent Clearly  these are  all possibilities.
 
 For item (2): Since $g_m$ is a permutation on $2^{2m}$,
 it suffices to show the following.
 Let $\sigma\in 2^{\leq 2m}$, $\sigma$ admits at most five $g_m$-pre-images.
 
 \noindent If $|\sigma|\leq m$, then all possible
  $g_m$-pre-images of $\sigma$ are
  (depending on which cases of $g_m$ happens)
  $$
  0^{|\sigma|}, \sigma,\sigma-1.
  $$
  
  \noindent If $|\sigma|>m$,
  let $\sigma = \sigma_x\sigma_y$ where $\sigma_x\in 2^m$, then all possible $g_m$-pre-images are
  $$
  (0^m,\sigma_x\uhr |\sigma_y|), (\sigma_x,\sigma_y), (\sigma_x-1,\sigma_y),
  (\sigma_x,\sigma_y-1)\text{ and }
   (\sigma_x-1,\sigma_y-1).
  $$

 \end{proof}

From now on, fix a $k\geq 4$
($k=4$ will be enough). 
For $a_0,a_1\in 2$, write $a_0\Delta a_1$
for the XOR of $a_0,a_1$.
For $a\in 2^k$ and $\sigma\in 2^n$
 (where $n\geq k$), let $ a\xor \sigma$ be the bitwise $\xor$ 
 of $a,\sigma$:  
 $$(a(1)\xor\sigma(1))\cdot... \cdot (a(k)\xor\sigma(k))\cdot 
 \sigma(k+1)\cdots\sigma(n) .$$
  For $z =  z_1 z_2 \cdots $ where $z_i =  x_i y_i $ where $x_i,y_i\in 2^{m_i}$,
 let 
 $$h_a(z) =  (x_1\xor a, y_1)\cdot (x_2\xor a, y_2)\cdots.
 $$
 Define
 $$
 f_a(z):= (g\circ h_a)(z).
 $$
 Obviously $f_a$ is computable.
 Since every string admits at most five $g$-pre-image 
 (see Claim \ref{maxKclaim0}), we have
 \begin{claim}\label{maxKclaim1}
 For every $a\in 2^k$,
 every string admits at most five $f_a$-pre-image.
 Therefore, $z\leq_K f_a(z)$ and $z\leq_C f_a(z)$ for all $z\in 2^\omega$.
 \end{claim}

 Thus, to prove Theorem \ref{maxKth0}, it remains to prove
 the following proposition:
 let $\Delta_K(\sigma):= K(\sigma)-|\sigma|$.
 \begin{proposition}\label{maxKlem0}
 Let $z\in 2^\omega$.
 If $f_a(z)\leq_K z$ for all $a\in 2^k$, then we have
 $$\Delta_K(z\uhr n)\rightarrow-\infty.$$ 
 \end{proposition}
We also have an analog of Proposition \ref{maxKlem0} with $K$-complexity
replaced by $C$-complexity:
  \begin{proposition}\label{maxKlem1}
 Let 1-random  real $z\in 2^\omega$.
 Then $z<_C f_a(z)$ for all $a\in 2^k$.
 \end{proposition}
 Before proving Proposition \ref{maxKlem0} and \ref{maxKlem1},
 we use them to derive some of the main results.
 \begin{proof}[Proof of Theorem \ref{maxKth0} assuming Proposition \ref{maxKlem0}]
 For any $z\in 2^\omega$, if $z<_K f_a(z)$ for some $a\in 2^k$, then we are done.
 We already have $z\leq_K f_a(z)$, so the only remaining case is
 $f_a(z)\leq_K z$ for all $a\in 2^k$. In this case the conclusion follows by Proposition \ref{maxKlem0}
 and noticing that for any 1-random real $\h z$,
 $\Delta_K(\h z\uhr n)\geq O(1)$
 (so the $K$-degree of $z$ is strictly below that of any 1-random real).
 
 \end{proof}
 
 \begin{proof}[Proof of Theorem \ref{maxKth1} assuming Proposition \ref{maxKlem0}]
 If $z$ is 1-random real, then $$\Delta_K(z\uhr n)\geq O(1).$$
 Thus, by Proposition \ref{maxKlem0}, it is not the case: $f_a(z)\leq_K z$ 
 for all $a\in 2^k$.
 Since $z\leq_K f_a(z)$ for all $a\in 2^k$,
 we have $z<_K f_a(z)$ for some $a\in 2^k$. 
 \end{proof}

 \begin{proof}[Proof of Theorem \ref{maxKthm:relative} assuming Proposition \ref{maxKlem0}, \ref{maxKlem1}]
 By Claim \ref{maxKclaim0} item (1), we have,
 for all $a\in 2^k$ and all real $x$,
 $x$ is 1-random relative to $d$ iff $f_a(x)$ is 1-random relative to $d$;
 and $x\equiv_{tt} f_a(x)$.
Therefore, given a 1-random relative to $d$ real $x$,
 let (by Proposition \ref{maxKlem0}, \ref{maxKlem1}) $a\in 2^k$ so that $x<_K f_a(x)$
 and $x<_C f_a(x)$.
 
 On the other hand, 
 for each $j\in\omega$, $f_a$ give rise to a permutation on 
 $2^{m_{\leq j}}$ (where $m_{\leq j}:= \sum_{i\leq j}2m_i$).
 Therefore, $$\big|K(x\uhr m_{\leq j}) - K(f_a(x)\uhr m_{\leq j})\big|,
 \big|C(x\uhr m_{\leq j}) - C(f_a(x)\uhr m_{\leq j})\big| = O(1).$$
  Thus, we are done for the first part.
The ``In particular" part follows directly.

 \end{proof}
 
  It remains to prove Proposition \ref{maxKlem0} and \ref{maxKlem1}.
 
 \subsection{Proof of Proposition \ref{maxKlem0}}
 \label{maxKsecverify}
 Fix  $z =  z_1 z_2 \cdots$ where $z_i =  x_i y_i$ where $x_i,y_i\in 2^{m_i}$.
 We say  block $i$ is \emph{rooted} iff for some $a\in 2^k$,
 $x_i= a\cdot 0^{m_i-|a|}$ (so $x_i\xor a = 0^{m_i}$).
 We firstly establish the key observation (\ref{maxKkey}) of complexity 
 deficiency 
 \begin{lemma}[Deficiency at rooted block]
 \label{maxKlemdeficiency}
If
 the $i^{th}$ block is rooted, then we have:
 $K(z_{\leq i})\leq K(z_{<i})+O_{z,k}(1)$.
 \end{lemma}
 \begin{proof}
 Since $(x_i,y_i)$ is rooted, say $x_i\xor a=0^{m_i}$, so (by definition of $f_a$) 
 $$f_a(z_{\leq i}) = f_a(z_{<i})\cdot g_m(x_i\xor a, y_i )
  = f_a(z_{<i})\cdot(y_i,0^{m_i}).$$
 Since $g_m$ is a permutation
 (so $f_a$ restricted on $2^{|z_{<i}|}$
 is also a permutation),
 we have $$
K( f_a(z_{<i}))= K(z_{<i}) +O(1).
 $$
 Now we have:
 \begin{align}\nonumber
 &K(z_{\leq i})\\ \nonumber
 \text{by Claim \ref{maxKclaim1} }     \leq& K(f_a(z_{\leq i}))+O(1)\\ \nonumber
  =&  K\bigg(f_a(z_{<i})\cdot (y_i,0^{m_i})\bigg) 
  \\ \nonumber
  \leq& K\bigg(f_a(z_{<i})\cdot y_i\bigg)+O(1)\\ \nonumber 
 \text{by 
 $f_a(z)\leq_K z$\ \ \ \ } \leq & K( z_{<i} \cdot  x_i )+O_z(1)\\ \nonumber
 \leq& K( z_{<i})+O_{z,k}(1).
 \end{align}

 Thus we are done proving Lemma \ref{maxKlemdeficiency}.
 
 \end{proof}

 Using 
 Lemma \ref{maxKlemdeficiency}, we can derive how deficiency accumulates 
 at  rooted blocks.
 \begin{lemma}
 [Deficiency accumulation]
 \label{maxKlemdeficiency2}
 Let $i<j$ and suppose there are no rooted blocks
 with index  in $(i,j)$. 
 We have:
 \begin{align}\label{maxKdeficiencybeforerooted2}
K(z_{<j}|z_{\leq i})&\leq  |z_{(i,j)}|+ \frac{3}{2}m_j- 2^{k-2}(m_j-m_i)+O_k(1).
\end{align}
If the $j^{th}$ block is rooted, then we have
\begin{align}\label{maxKdeficiencyaccumulate}
\Delta_K(z_{\leq j})+2^{k-2}m_j&\leq \Delta_K(z_{\leq i})+ 2^{k-2}m_i
 -\frac{m_j}{2}+O_{z,k}(1).
\end{align}
Thus, 
\begin{align}\label{maxKdeficiencyaccumulation}
\lim_{j\rightarrow\infty, j^{th} \text{ block is rooted}}\Delta_K(z_{\leq j})+2^{k-2}m_j = -\infty.
\end{align}
 \end{lemma}
\begin{proof}   
 Let's look at $K(z_{<j})- K(z_{\leq i}) \leq  K(z_{<j}|z_{\leq i})+O(1)$:
 to describe $z_{<j}$ given $z_{\leq i}$, we only need
 \begin{itemize}
 \item a description of $j$;
 \item a description of $z_{(i,j)}$ (given $j,z_{\leq i}$).
 
 \end{itemize} 
 That is, 
 \begin{align}\label{maxKeq00}
 K(z_{<j}|z_{\leq i})\leq   K(z_{(i,j)}|j,z_{\leq i}) + K(j)+O(1).
 \end{align}
 For description of $j$, we simply take 
 \begin{align}\label{maxKeq1}
 K(j)\leq \frac{3}{2}\log_2 j+O(1)\leq \frac{3}{2}m_j+O(1).
 \end{align}
 For description of $z_{(i,j)}|j,i$, we take advantage of: blocks in  $z_{(i,j)}
 $ are non rooted. This provides a (conditional) description smaller than $|z_{(i,j)}|$.

 Let $\NR_{ij}$ denote the set of string $\rho\in 2^{|z_{(i,j)}|}$
 such that the corresponding blocks in $\rho$ are non rooted \footnote{
 That is, if $\rho =  \h x_{i+1} \h y_{i+1} \cdots  \h x_{j-1} \h y_{j-1} $
 with $x_{\h i},y_{\h i}\in 2^{m_{\h i}}$, then 
 for every $a\in 2^k$, every $\h i\in (i,j)$, $\h x_{\h i}\ne a\cdot 0^{m_{\h i}-k}$.
 }.
 Since $\NR_{ij}$ is a  finite set 
 whose index is computable given $i,j$, we have for every $\sigma\in \NR_{ij}$:
\begin{align}\label{maxKeqnonrootcomplexity}
 K(\sigma|i,j)\leq& \log_2|\NR_{ij}|+O(1)\\ \nonumber
 =&  |z_{(i,j)}| +\sum_{\h i\in (i,j)} \log_2(1-\frac{2^k}{2^{m_{\h i}}})+O(1)\\ \nonumber
 \leq &|z_{(i,j)}|- \sum_{\h i\in (i,j)}\frac{2^k}{2^{m_{\h i}}}+O(1)\\ \nonumber
\text{recall $m_{\h i} = [\log_2\h i]\vee k+1$\ \ \ \ } \leq& |z_{(i,j)}|- \sum_{\h i\in (i,j)} \frac{2^k}{ 2\h i}+O_k(1)\\ \nonumber
 \leq& |z_{(i,j)}|- 2^{k-1}(\log(j)- \log(i))+O_k(1)\\ \nonumber
 \leq& |z_{(i,j)}|- 2^{k-2}(\log_2(j)- \log_2(i))+O_k(1)\\ \nonumber
 \leq&|z_{(i,j)}| -2^{k-2}(m_j-m_i)+O_k(1).
 \end{align}
In summary of (\ref{maxKeq00})(\ref{maxKeq1}) and (\ref{maxKeqnonrootcomplexity}),
we have
\begin{align}\label{maxKeqdeficiencybeforrooted}
K(z_{<j}|z_{\leq i})&\leq \frac{3}{2}m_j+O(1)+|z_{(i,j)}|- 2^{k-2}(m_j-m_i)+O_k(1)\\ \nonumber
& = |z_{(i,j)}|+ \frac{3}{2}m_j- 2^{k-2}(m_j-m_i)+O_k(1).
\end{align}
This verifies (\ref{maxKdeficiencybeforerooted2}).
Moreover, 
suppose $j^{th}$ block is rooted,
combine with  Lemma \ref{maxKlemdeficiency}, we have
\begin{align}\nonumber
\Delta_K(z_{\leq j}) - \Delta_K(z_{\leq i})
 &\leq K(z_{<j})+O_{z,k}(1)-K(z_{\leq i}) - (|z_{\leq j}|-|z_{\leq i}|)\\ \nonumber
 &\leq  K(z_{<j}|z_{\leq i})+O_{z,k}(1)  -|z_{(i,j)}| - 2m_j\\ \nonumber
 &\overset{(\ref{maxKeqdeficiencybeforrooted})}{\leq}
 -2^{k-2}(m_j-m_i)-\frac{1}{2}m_j+O_{z,k}(1)\\ \nonumber
 \Rightarrow 
 \Delta_K(z_{\leq j})+2^{k-2}m_j&\leq \Delta_K(z_{\leq i})+ 2^{k-2}m_i
 -\frac{m_j}{2}+O_{z,k}(1).
\end{align}
In a word, 
\begin{align}\nonumber
&\text{the deficiency loss due to description of $j$ (namely $\leq \frac{3}{2}m_j$)}\\ \nonumber
&\text{is compensated by 
  the deficiency increase
  (namely $ 2m_j$) }\\ \nonumber
  &\text{due to the rootedness 
of the $j^{th}$ block.}
\end{align}
The last part (\ref{maxKdeficiencyaccumulation})
simply follows by noticing $m_j\rightarrow\infty$.
Thus, we are done.

 \end{proof}
 
 Now we are ready to show $\Delta_K(z\uhr n)\rightarrow-\infty$.
  Let $n\in\omega$ be sufficiently large and $j$ be the block that $n$ locates
  (so $j, m_j$ are also sufficiently large).
 i.e., $$\sum_{\h i<j}2m_{\h i}< n \leq \sum_{\h i\leq j}2m_{\h i}.$$
 Let $i$ be the (index of)  latest rooted block of $z$ before the $j^{th}$ block,
 or $i=1$ if there is no such  block.
 We will use $K(z_{\leq i})$ to control $K(z\uhr n)$.
 Let $m$ be the coordinate of $n$ in the $j^{th}$
 block, namely $m:= n-\sum_{\h i<j}2m_{\h i}$.
 Note that given $z_{<j}$ to describe $z\uhr n$,
 we need a description of $m$ and a (conditional on $m$) description of 
 the last $m$ bits of $z\uhr n$
 (which in total is bounded by $m+\frac{3}{2}\log_2m$).
Therefore,
  \begin{align}\label{maxKeq54}
 K(z\uhr n) - K(z_{\leq i}) &\leq K(z\uhr n|z_{\leq i})+O(1)\\ \nonumber
 &\leq K(z\uhr n|z_{<j})+ K(z_{<j}|z_{\leq i})+O(1)\\ \nonumber
\text{by Lemma \ref{maxKlemdeficiency2}
(\ref{maxKdeficiencybeforerooted2})}\ \ \ \  &\leq m+\frac{3}{2}\log_2 m +  |z_{(i,j)}|+ \frac{3}{2}m_j- 2^{k-2}(m_j-m_i)+O_k(1)\\ \nonumber
 \text{since $m\leq 2m_j$\ \ \ \ }
 &\leq m+ \frac{3}{2}\log_2 m_j+ |z_{(i,j)}|+ \frac{3}{2}m_j- 2^{k-2}(m_j-m_i)+O_k(1)\\ \nonumber
\text{re-organize}\ \ \  &=|z_{(i,j)}|+m+ \frac{3}{2}m_j- 2^{k-2}(m_j-m_i)+  \frac{3}{2}\log_2 m_j+O_k(1)
 \end{align}
 Thus, since  $z\uhr n$ is $(|z_{(i,j)}|+m)$-longer than $z_{\leq i}$,
   \begin{align}\label{maxKeq55}
\Delta_K(z\uhr n) -\Delta_K(z_{\leq i})& = 
 K(z\uhr n) - K(z_{\leq i}) -(|z_{(i,j)}|+m)\\ \nonumber
 &\leq  \frac{3}{2}m_j- 2^{k-2}(m_j-m_i)+  \frac{3}{2}\log_2 m_j+O_k(1)\\ \nonumber
 \text{since $k\geq 4$
 and $j$ sufficiently large\ \ \ \ }&\leq 
 2^{k-2}m_i-\frac{1}{4}m_j
 \end{align}
 Therefore,
    \begin{align}\nonumber
    \Delta_K(z\uhr n)\leq \Delta_K(z_{\leq i})+ 2^{k-2}m_i-\frac{1}{4}m_j\\ \nonumber
    \text{by Lemma \ref{maxKlemdeficiency2}}
    (\ref{maxKdeficiencyaccumulation}) \ \ \ \  \leq O_{z,k}(1)-\frac{1}{4}m_j.
    \end{align}
Thus, letting  $n\rightarrow\infty$ (so $m_j, j\rightarrow\infty$),
we have
 $\Delta_K(z\uhr n)\rightarrow-\infty$.
 Thus we are done
 proving Proposition \ref{maxKlem0}.
 
 \subsection{Proof of Proposition \ref{maxKlem1}}
Fix $a\in 2^k$ and suppose otherwise $f_a(z)\leq_C z$.
Note that
  $z$ must admit infinitely many rooted blocks (witnessed by $a$)
since the probability of avoiding all rooted block (witnessed by $a$)
is $$\prod_i (1-2^{-m_i}) = 0 .$$ 
The same argument as Claim \ref{maxKlemdeficiency} gives: for every $a\in 2^k$ and every
rooted block $i$ witnessed by $a$,
 \begin{align}\label{maxKeq50}
 &C(z_{\leq i}) 
  \leq  C( z_{<i})+O_{z,k}(1)  
 \leq  |z_{\leq i}|-2m_i+O_{z,k}(1).
 \end{align}
  \def\m{m}
However, let $m_{\leq i}:= \sum_{j\leq i}2m_j$, 
and consider the set 
$V_i\subseteq 2^{m_{\leq i}}$ of $\rho$
such that $$C(\rho)\leq m_{\leq i}-2m_i+O_{z,k}(1).$$
Obviously, 
$\m(V_i)\leq O_{z,k}(2^{-2m_i})$.
But $$\sum_i 2^{-2m_i} \leq O_{z,k}(1)\cdot\sum_{i\geq 1} 1/i^2<\infty.$$ 
Thus, $(V_i:i\in\omega)$ is a Solovay test.
By (\ref{maxKeq50}), $z\in [V_i]$ for infinitely many $i\in\omega$,
so $z $ is not 1-random, a contradiction. 
Thus we are done proving Proposition \ref{maxKlem1}.

 \subsection{Proof of Theorem \ref{maxKth:maxCdegree}  }
  Suppose otherwise $f_a(z)\leq_C z$ for all $a\in 2^k$.
  We show that the $C$-degree of any 1-random real is strictly 
  above  that of $z$.   
  Let $\h z\in 2^\omega$ be 1-random,
 by Miller and Yu \cite{MY08} Theorem 7.1, for all $n\in\omega$,
\begin{align}\nonumber
 C(\h z\uhr n)&\geq n-K(n)-O(1)\\ \nonumber
 &\geq n-\frac{3}{2}\log_2 n-O(1).
  \end{align}
  Thus, it suffices to show that 
  \begin{align}\label{maxKeq57}
  C(z\uhr n)-n+\frac{3}{2}\log_2 n\rightarrow-\infty.
  \end{align}
  
  We first connect $C$-complexity of $z$ with $K$-complexity of $z$
  and obtain the following Claim.
  \begin{claim}\label{maxKclaim2}
 We have for every $a\in 2^k$ and every $n\in\omega$,
 $$K(f_a(z)\uhr n)\leq K(z\uhr n) +O(\log\log n)+O_z(1).$$ 
  \end{claim}
  \begin{proof}
  We use  the following Solovay's formula  given by Bauwens \cite{bauwens2016relating} Theorem 2:
$$
K(\rho)
=C(\rho)+C(C(\rho))
 +O\bigl(C(C(C(\rho)))+1\bigr).  
 $$
  Let $\rho:= z\uhr n $ and $\sigma = f(z)\uhr n$.
  So $ |C(\rho) - C(\sigma)| \leq O_z(1)$.
  Therefore $$ C(C(\rho))-C(C(\rho)) \leq O_z(1).$$ 
 On the other hand, $C(\rho)\leq n+O(1)$
 so $C(C(\rho))\leq \log_2 n+O(1)$
 and $$C(C(C(\rho))),C(C(C(\sigma)))\leq \log_2\log_2 n+O_z(1).$$
 Thus,
 \begin{align}\nonumber
  &K(z\uhr n) -K(f_a(z)\uhr n)
 \\ \nonumber
  \leq& \bigg(C(\rho) - C(\sigma)\bigg)+ 
  \bigg(C(C(\rho))- C(C(\sigma))\bigg)
  + O\bigg(C(C(C(\rho)))+ C(C(C(\sigma)))+1\bigg)\\ \nonumber
  \leq& O(\log_2\log_2 n)+O_z(1).
 \end{align}

  \end{proof}
  For every $n\in\omega$ in the $i^{th}$ block, 
  since $n\leq 2i\log_2 i$, we have $$\log\log n\leq \log m_i+O(1).$$
Now similarly as Lemma \ref{maxKlemdeficiency},
for every rooted block $i$ witnessed by $a$,
 \begin{align}\label{maxKeq53}
 &K(z_{\leq i})\\ \nonumber
 \text{by Claim \ref{maxKclaim1} \ \ \ \ }     \leq& K(f_a(z_{\leq i}))+O(1)\\ \nonumber
  =&  K\bigg(f_a(z_{<i})\cdot (y_i,0^{m_i})\bigg) 
  \\ \nonumber
  \leq& K\bigg(f_a(z_{<i})\cdot y_i\bigg)+O(1)\\ \nonumber 
 \text{by 
Claim \ref{maxKclaim2}\ \ \ \ } \leq & K( z_{<i} \cdot  x_i )+O(\log m_i)+O_z(1)\\ \nonumber
 \leq& K( z_{<i})+O(\log m_i)+O_{z,k}(1).
 \end{align}
 
  Next, similarly as Lemma \ref{maxKlemdeficiency2},
  suppose $i<j$ and  $j$ is a rooted block
  and there is no rooted block in $(i,j)$,
 using Lemma \ref{maxKlemdeficiency2} (\ref{maxKdeficiencybeforerooted2})
 (which does not concern $f_a(z)\leq_K z$), we have
 \begin{align}\nonumber
&\Delta_K(z_{\leq j}) - \Delta_K(z_{\leq i})
\\ \nonumber
  \overset{(\ref{maxKeq53})}{\leq}&
  K(z_{<j})+O(\log m_j)+O_{z,k}(1)-K(z_{\leq i}) - (|z_{\leq j}|-|z_{\leq i}|)\\ \nonumber
\leq& K(z_{<j}|z_{\leq i})-|z_{(i,j)}| - 2m_j +O(\log m_j)+O_{z,k}(1)\\ \nonumber
  \overset{ (\ref{maxKdeficiencybeforerooted2})}{\leq}&
  \frac{3}{2}m_j -2^{k-2}(m_j-m_i)-2m_j+O(\log m_j)+O_{z,k}(1)\\ \nonumber
  \leq& -2^{k-2}(m_j-m_i)-\frac{1}{4}m_j+O_{z,k}(1).
  \\ \nonumber
    \Rightarrow 
 \Delta_K(z_{\leq j})&+2^{k-2}m_j  \leq \Delta_K(z_{\leq i})+ 2^{k-2}m_i
 -\frac{1}{4}m_j+O_{z,k}(1)\\ \nonumber
\end{align}
 Thus,  we have
 \begin{align}
 \label{maxKeq56}
 \text{the set } (\Delta_K(z_{\leq j})+2^{k-2}m_j: j^{th}\text{ block is rooted})  \text{ is bounded above}.
 \end{align}

  Finally, let $n\in\omega$ be sufficiently large (depending on $k$) and $j$ be the block that $n$ locates
  (so $j, m_j$ are also sufficiently large depending on $k$).
 Let $i$ be the (index of)  latest rooted block of $z$ before the $j^{th}$ block,
 or $i=1$ if there is no such  block.
  Let $m$ be the coordinate of $n$ in the $j^{th}$
 block, namely $m:= n-  m_{<j}$.
 Therefore, similarly as (\ref{maxKeq55})
    \begin{align}\nonumber
\Delta_K(z\uhr n) -\Delta_K(z_{\leq i})& = 
 K(z\uhr n) - K(z_{\leq i}) -(|z_{(i,j)}|+m)\\ \nonumber
 &\leq  \frac{3}{2}m_j- 2^{k-2}(m_j-m_i)+  \frac{3}{2}\log_2 m_j+O_k(1)\\ \nonumber
 \text{since $k $ 
 and $j$ sufficiently large\ \ \ \ }&\leq 
 2^{k-2}m_i-2m_j
 \end{align}
 Therefore,
    \begin{align}\label{maxKeq222}
    \Delta_K(z\uhr n)\leq &\Delta_K(z_{\leq i})+ 2^{k-2}m_i-2m_j
    \\ \nonumber
    \overset{(\ref{maxKeq56})}{\leq}& O_{z,k}(1) -2m_j
    \leq O_{z,k}(1)-\frac{3.5}{2}\log_2 n.
    \end{align}
  
  Now we can verify (\ref{maxKeq57}):
   since $C(\sigma)\leq K(\sigma)+O(1)$ for all $\sigma\in 2^{<\omega}$,
  we have 
\begin{align}\nonumber
  C(z\uhr n)-n+\frac{3}{2}\log_2 n&\leq 
  \Delta_K(z\uhr n)+O(1)+\frac{3}{2}\log_2 n
  \\ \nonumber
  &\overset{(\ref{maxKeq222})}{\leq} O_{z,k}(1)-\frac{3.5}{2}\log_2 n+\frac{3}{2}\log_2 n\rightarrow-\infty.
\end{align}
   This verifies (\ref{maxKeq57}) and we are done proving Theorem \ref{maxKth:maxCdegree}.

  \section*{Acknowledgement}
 This research is supported by  National Natural Science Foundation of China grant 12471002.
 
 Author would like to thank Xuanheng Zhao for pointing to us that the proof for ``no maximal $K$-degree"
 can be adapted to a proof for ``no maximal $C$-degree".
 
\textbf{ The proofs of  main results in this paper
 are provided by AI}.
  Author read,  carefully checked, somewhat simplified  the proof given by AI
  and take full responsibility.
 
 \bibliographystyle{amsplain}
\bibliography{C:/6+1/Draft/bibliographylogic}

\end{document}